\documentclass[11pt]{article}

\usepackage{amsmath, amssymb, amsthm}
\usepackage{graphicx}
\usepackage{hyperref}
\usepackage{geometry}
\usepackage{hyperref}
\usepackage{epigraph}
\usepackage{amssymb}
\usepackage{soul}
\usepackage{amssymb}
\usepackage{tikz}
\usepackage[font=small,skip=5pt]{caption}
\usepackage{xcolor}
\usepackage{lineno}

\newtheorem{theorem}{Theorem}
\newtheorem{lemma}[theorem]{Lemma}
\newtheorem{corollary}[theorem]{Corollary}

\newtheorem{conjecture}[theorem]{Conjecture}

\newcommand{\force}[2]{(#1 \rightarrow #2)}
\newcommand{\zforce}[2]{(#1 \rightarrow #2)_z}
\newcommand{\psdforce}[2]{(#1 \rightarrow #2)_{+}}

\newcommand{\F}{\mathcal{F}}

\DeclareMathOperator{\Term}{Term}

\makeatletter
\newcommand{\smallCases}[1]{{\arraycolsep 3.2pt\Big\{\mkern-3mu\let\@nomath\@gobble\footnotesize\begin{array}{ll}#1\end{array}\mkern-4mu}}
\tikzset{super thick/.style={line width=2pt}}

\title{A Characterization of Claw-Free Graphs \\ using Zero Forcing Invariants}

\author{
    Randy Davila\thanks{Corresponding author: \texttt{rrd6@rice.edu}} \\
    \small Department of Computational Applied \\ 
    \small Mathematics \& Operations Research \\ 
    \small Rice University \\ 
    \small \texttt{rrd6@rice.edu} 
    \and
    Houston Schuerger \\
    \small Department of Mathematics \\ 
    \small University of Texas Permian Basin \\ 
    \small \texttt{schuerger\_h@utpb.edu}
    \and
    Ben Small \\
    \small University Place, WA \\
    \small \texttt{bentsm@gmail.com}
}

\date{}
\begin{document}

\maketitle

\begin{abstract}
We prove that the \emph{standard zero forcing number} $Z(G)$ and the \emph{positive semidefinite zero forcing number} $Z_+(G)$ are equal for all claw-free graphs $G$. This result resolves a conjecture proposed by the computer program \emph{TxGraffiti} and highlights a connection between these graph invariants in claw-free structures. As a corollary, we show that a graph $G$ is claw-free if and only if every induced subgraph $H \subseteq G$ satisfies $Z(H) = Z_+(H)$. 
\end{abstract}

\epigraph{``Try to use the claw, but you don't need to use the claw''}{-- Baloo the Bear, \emph{The Jungle Book}}

{\small \textbf{Keywords:} \emph{claw-free graphs}; \emph{positive semidefinite zero forcing number}; \emph{TxGraffiti}; \emph{standard zero forcing number}} \\
\indent {\small \textbf{AMS subject classification: 68R10}}

\section{\textbf{INTRODUCTION}}\label{sec:introduction}
Let $G$ be an undirected, simple, and finite graph with vertex set $V(G)$, and let $V(G) = B \cup W$ represent a partition of the vertices into a set of blue vertices $B$ and a set of white vertices $W$, with time $t = 0$ initialized. A \emph{forcing game} on a graph is a process in which a player strategically applies a \emph{color change rule} to propagate blue, or white, vertices through the graph. A color change rule specifies a condition under which a vertex $v \in V(G)$ can be \emph{forced} to change color (either from white to blue, or from blue to white).  In this paper we only consider those rules allowing blue vertices to force color changes on white vertices, denoted $\force{u}{v}$, where $u$ is a blue vertex and $v$ is a white vertex changing to be blue. A \emph{valid force} under a given rule is any pair $\force{u}{v}$ where $u \in B$ and $v \in W$ satisfies the condition specified by the rule. At each time step $t$, a player selects a subset of valid forces, ensuring that no white vertex is forced by more than one vertex. If $\Delta B_t$ denotes the set of newly blue vertices resulting from these forces, the sets are then updated as $B \gets B \cup \Delta B_t$ and $W \gets W \setminus \Delta B_t$, and the time is incremented as $t \gets t + 1$.  This process is repeated until no further changes occur, resulting in a \emph{terminal time step}, denoted by $\tau$. If all vertices in $V(G)$ are blue at this terminal time step, the initial set $B$ is called a \emph{forcing set} of $G$ under the given rule. The minimum cardinality of such a forcing set defines a graph invariant associated with the rule.

Two of the most studied color change rules are the \emph{standard zero forcing rule} and the \emph{positive semidefinite zero forcing rule} (or simply \emph{standard forcing rule} and \emph{positive semidefinite forcing rule} for brevity). The \emph{standard forcing rule} allows a vertex $u \in B$ with exactly one neighbor $v \in W$ to \emph{standard force} $v$, denoted $\zforce{u}{v}$. Forcing sets under this rule are called \emph{standard forcing sets}, and the minimum cardinality of such a set is the \emph{standard forcing number}, $Z(G)$. The \emph{positive semidefinite forcing rule} generalizes the standard rule by allowing a blue vertex $u \in B$ to force a white vertex $v \in W$ if $v$ is the only neighbor of $u$ in its component of the induced subgraph $G[W]$, denoted $\psdforce{u}{v}$. Forcing sets under this rule are called \emph{positive semidefinite forcing sets} and the minimum cardinality of such a set is the \emph{positive semidefinite forcing number}, $Z_+(G)$.

Both $Z(G)$ and $Z_+(G)$ were initially motivated by applications in linear algebra, particularly in the problem of determining the \emph{maximum nullity} (equivalently, the \emph{minimum rank}) over symmetric matrices described by a graph. Specifically, $Z(G)$ provides the upper bound $M(G) \leq Z(G)$ for the maximum nullity $M(G)$ of symmetric matrices whose off-diagonal zero-nonzero pattern corresponds to the adjacency structure of $G$ (see~\cite{AIM2008}).  Similarly, $Z_+(G)$ bounds the \emph{maximum positive semidefinite nullity}, $M^+(G)$, which is the largest nullity among symmetric positive semidefinite matrices whose off-diagonal zero-nonzero pattern matches $G$ (see~\cite{Barioli2013}). 

It is straightforward to observe that $Z_+(G) \leq Z(G)$ for any graph $G$, as the positive semidefinite forcing rule has more flexible criteria for valid forces than the standard forcing rule. Furthermore, the gap between these parameters can be arbitrarily large: for any constant $C > 0$, there exists a graph $G$ such that $Z(G) - Z_+(G) > C$. For example, in any tree $T$, the positive semidefinite forcing number satisfies $Z_+(T) = 1$ regardless of the structure of $T$, while the standard forcing number $Z(T)$ can grow arbitrarily large with the maximum degree or the number of strong support vertices in $T$. Conversely, certain graphs, such as paths and diamond-necklaces, satisfy $Z_+(G) = Z(G)$ (see~\cite{Hogben2022Inverse}). Indeed, the conjecturing program \emph{TxGraffiti} (see~\cite{Davila2024} for a description of the program) conjectured the following:
\begin{conjecture}[\emph{TxGraffiti} -- confirmed]\label{conj:main} 
If $G$ is a claw-free graph, then $Z_+(G) = Z(G)$. 
\end{conjecture}

Motivated by Conjecture~\ref{conj:main}, which we confirm in this paper, we consider a broader question: for which families of graphs does the equality $Z_+(G) = Z(G)$ hold universally, even for all induced subgraphs? This question aligns with classic themes in graph theory, such as the study of perfect graphs, where the equality of clique number and chromatic number holds for every induced subgraph~\cite{west_introduction_2000}. To formalize this, we define a graph $G$ to be \emph{$(Z_+, Z)$--perfect} if $Z_+(H) = Z(H)$ for every induced subgraph $H \subseteq G$. As we demonstrate, claw-free graphs offer a complete characterization of $(Z_+, Z)$--perfect graphs, unifying and extending prior results on the relationship between $Z(G)$ and $Z_+(G)$ in claw-free graphs, such as $Z(G) = Z_+(G)$ for \emph{line graphs} (see~\cite{Fallat2016}) and the bounds $Z(G) \leq Z_+(G) \leq 2Z(G)$ (see~\cite{Wang2019}).

As part of proving Conjecture~\ref{conj:main}, we also establish in this paper an interesting result about positive semidefinite forcing sets.  A central focus in \cite{PIPs} is what can be said when forcing sets are vertex cuts.  Utilizing the concept of path bundles introduced in \cite{PIPs}, we explore the opposite and show in Lemma~\ref{lem:connected} that for any graph, there is a minimum positive semidefinite forcing set of that graph whose complement is connected. This connectivity plays a key role in the proof of Conjecture~\ref{conj:main}.

The remainder of this paper is organized as follows. In Section~\ref{sec:notation-and-terminology}, we provide the necessary definitions and preliminaries, including \emph{relaxed chronologies} and \emph{path bundles}. Section~\ref{sec:the-lemmas} provides lemmas regarding positive semidefinite forcing needed to prove our main result. Section~\ref{sec:the-theorem} contains the proof of our main theorem and its corollary. Finally, in Section~\ref{sec:conclusion}, we give concluding remarks.


\section{\textbf{RELAXED CHRONOLOGIES AND PATH BUNDLES}}\label{sec:notation-and-terminology}
For any common graph terminology or notation not explicitly defined in this paper, we refer the reader to~\cite{west_introduction_2000}, and for undefined zero forcing terminology~\cite{Hogben2022Inverse}. 
Let $S_+(G, B)$ denote the set of valid positive semidefinite forces $\psdforce{u}{v}$ when exactly $B$ is colored blue, where $u \in B$ and $v \in V(G)\setminus B$. A \emph{relaxed chronology}, denoted by $\F$, is a sequence consisting of the sets of positive semidefinite forces applied at each time step $t$ until a terminal state is reached. The elements of this sequence, denoted $F^{(t)}$, satisfy
\[
F^{(t)} \subseteq S_+(G, E_\F^{[t-1]}),
\]
where $E_\F^{[t-1]}$ is the set of blue vertices at the start of time step $t$. In addition, note that for each vertex $v$ which is not blue at time $t=0$, there exists a unique $t$ and a unique force $\psdforce{u}{v}$ such that $\psdforce{u}{v} \in F^{(t)}$. Finally, a special case of relaxed chronologies of particular importance occurs when one requires that exactly one force occurs during each time step. Such a relaxed chronology is called a \emph{chronological list of forces}. 
\begin{figure}[ht]
    \centering
    \begin{tikzpicture}[scale=1.4, every node/.style={circle, draw, minimum size=6mm, font=\small}, label/.style={shape=rectangle, draw=none, font=\small},initial blue/.style={node font=\bfseries,draw=blue,line width=1.8pt,fill=blue!30},forced blue/.style={draw=blue,line width=1pt,fill=blue!10}]
        \node[fill=white] (v1) at (-2, 0) {1};
        \node[initial blue] (v2) at (-1, 1) {2};
        \node[fill=white] (v3) at (-1, -1) {3};
        \node[initial blue] (v4) at (0, 0) {4};
        \draw (v1) -- (v2) -- (v3) -- (v1);
        \draw (v2) -- (v4) -- (v3);
        
        \node[fill=white] (v5) at (1, 0) {5};
        \node[initial blue] (v6) at (2, 1) {6};
        \node[fill=white] (v7) at (2, -1) {7};
        \node[fill=white] (v8) at (3, 0) {8};
        \draw (v5) -- (v6) -- (v7) -- (v5);
        \draw (v6) -- (v8) -- (v7);
        \draw (v4) -- (v5);
        \node[label] at (0.5, -1.65) {Time step $t=0$: Initial $B=\{2, 4, 6\}$};

        \node[fill=white] (w1) at (4, 0) {1};
        \node[initial blue] (w2) at (5, 1) {2};
        \node[forced blue] (w3) at (5, -1) {3};
        \node[initial blue] (w4) at (6, 0) {4};
        \draw (w1) -- (w2) -- (w3) -- (w1);
        \draw (w2) -- (w4) -- (w3);

        \node[forced blue] (w5) at (7, 0) {5};
        \node[initial blue] (w6) at (8, 1) {6};
        \node[fill=white] (w7) at (8, -1) {7};
        \node[fill=white] (w8) at (9, 0) {8};
        \draw (w5) -- (w6) -- (w7) -- (w5);
        \draw (w6) -- (w8) -- (w7);
        \draw (w4) -- (w5);
        \draw[->, very thick, blue] (w4) -- (w5); 
        \draw[->, very thick, blue] (w4) -- (w3); 
        \node[label] at (6.5, -1.65) {Time step $t=1$: $F^{(1)}=\{\psdforce{4}{3}, \psdforce{4}{5}\}$};

    \begin{scope}[yshift=0.28cm]
        \node[forced blue] (x1) at (-2, -4) {1};
        \node[initial blue] (x2) at (-1, -3) {2};
        \node[forced blue] (x3) at (-1, -5) {3};
        \node[initial blue] (x4) at (0, -4) {4};
        \draw (x1) -- (x2) -- (x3) -- (x1);
        \draw (x2) -- (x4) -- (x3);

        \node[forced blue] (x5) at (1, -4) {5};
        \node[initial blue] (x6) at (2, -3) {6};
        \node[forced blue] (x7) at (2, -5) {7};
        \node[fill=white] (x8) at (3, -4) {8};
        \draw (x5) -- (x6) -- (x7) -- (x5);
        \draw (x6) -- (x8) -- (x7);
        \draw (x4) -- (x5);
        \draw[->, very thick, blue] (x3) -- (x1); 
        \draw[->, very thick, blue] (x5) -- (x7); 
        \node[label] at (0.5, -5.65) {Time step $t=2$: $F^{(2)}=\{\psdforce{3}{1}, \psdforce{5}{7}\}$};

        \node[forced blue] (y1) at (4, -4) {1};
        \node[initial blue] (y2) at (5, -3) {2};
        \node[forced blue] (y3) at (5, -5) {3};
        \node[initial blue] (y4) at (6, -4) {4};
        \draw (y1) -- (y2) -- (y3) -- (y1);
        \draw (y2) -- (y4) -- (y3);

        \node[forced blue] (y5) at (7, -4) {5};
        \node[initial blue] (y6) at (8, -3) {6};
        \node[forced blue] (y7) at (8, -5) {7};
        \node[forced blue] (y8) at (9, -4) {8};
        \draw (y5) -- (y6) -- (y7) -- (y5);
        \draw (y6) -- (y8) -- (y7);
        \draw (y4) -- (y5);
        \draw[->, very thick, blue] (y6) -- (y8); 
        \node[label] at (6.5, -5.65) {Terminal time step $\tau=3$: $F^{(3)}=\{\psdforce{6}{8}\}$};
    \end{scope}
    \end{tikzpicture}
    \caption{An example of a relaxed chronology of forces made by applying the positive semidefinite forcing rule on two diamonds connected by a single edge. Each time step shows the propagation of blue vertices.}
    \label{fig:psd-forces-diamonds-grid}
\end{figure}
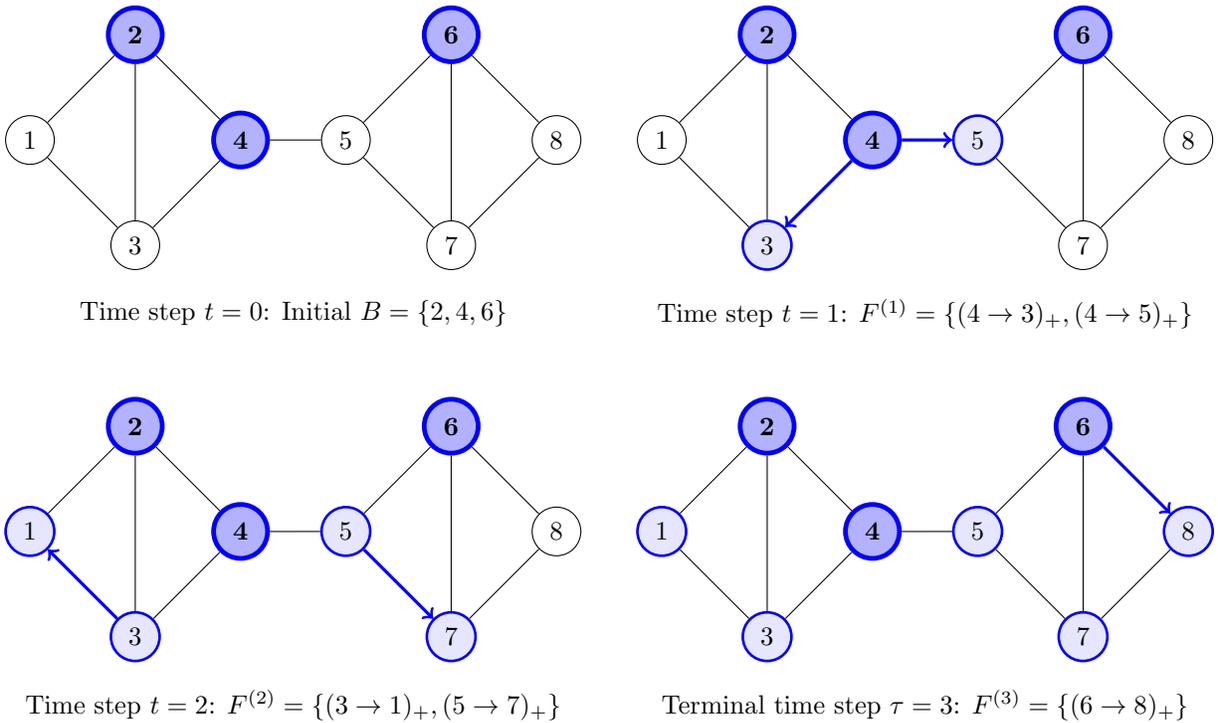

The relaxed chronology $\F$ governs the propagation of blue vertices over time, culminating in all vertices of $G$ being blue if the initial set $B$ is a positive semidefinite forcing set; see Figure~\ref{fig:psd-forces-diamonds-grid} for an example. The relaxed chronology $\F$ induces an \emph{expansion sequence}, $\big\{E_\F^{[t]}\big\}_{t=0}^{\tau}$, which tracks the progression of blue vertices over time. The sequence is defined recursively as:
\[
E_\F^{[0]} = B, \quad \text{and} \quad E_\F^{[t]} = E_\F^{[t-1]} \cup \{v : \psdforce{u}{v} \in F^{(t)} \text{ for some } u \in V(G)\},
\]
for all $t$ with $1 \leq t \leq \tau$. Here, $E_\F^{[t]}$ represents the set of blue vertices after time step $t$, with the propagation governed by the forces in $\F$.

For an induced subgraph $H \subseteq G$, the \emph{restriction} of a relaxed chronology $\F$ to $H$, denoted $\F|_H$, is the ordered collection of sets of forces $\F|_H = \{F^{(t)}|_H\}_{t=1}^{\tau}$. For each $t$ with $1 \leq t \leq \tau$, the set of restricted forces $F^{(t)}|_H$ is defined as:
\[
F^{(t)}|_H = \{\psdforce{u}{v} \in F^{(t)} : u, v \in V(H)\}.
\]
The restriction $\F|_H$ identifies all forces in $\F$ that occur entirely within the induced subgraph $H$. This concept is essential for analyzing the behavior of positive semidefinite forcing within specific substructures of a graph. By focusing on the restricted chronology $\F|_H$, we gain insight into how local propagation dynamics contribute to the global positive semidefinite forcing process.

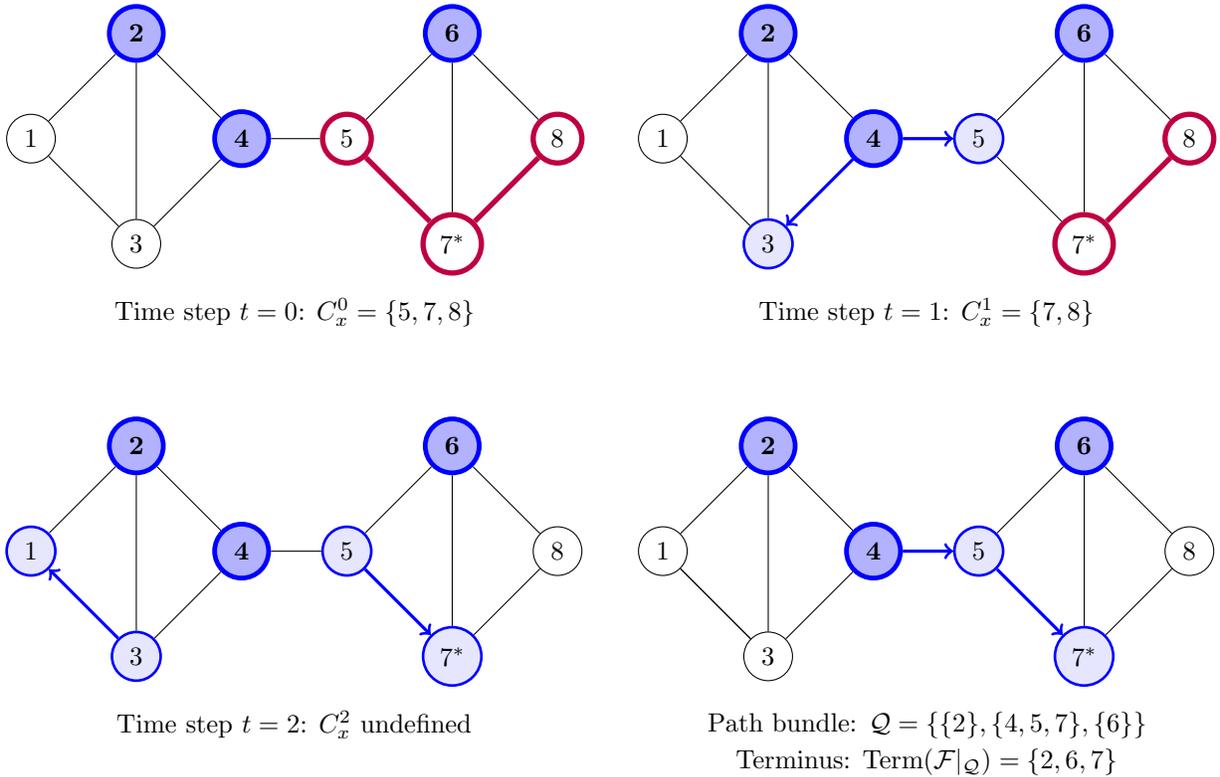
\begin{figure}[ht]
    \centering
    \begin{tikzpicture}[scale=1.4, every node/.style={circle, draw, minimum size=6mm, font=\small}, label/.style={shape=rectangle, draw=none, font=\small}, initial blue/.style={node font=\bfseries,draw=blue,line width=1.8pt,fill=blue!30},forced blue/.style={draw=blue,line width=1pt,fill=blue!10}]
        \node[fill=white] (v1) at (-2, 0) {1};
        \node[initial blue] (v2) at (-1, 1) {2};
        \node[fill=white] (v3) at (-1, -1) {3};
        \node[initial blue] (v4) at (0, 0) {4};
        \draw (v1) -- (v2) -- (v3) -- (v1);
        \draw (v2) -- (v4) -- (v3);

        \node[fill=white] (v5) at (1, 0) {5};
        \node[draw=purple, super thick, fill=white] (v5) at (1, 0) {5}; 
        \node[initial blue] (v6) at (2, 1) {6};
        \node[draw=purple, super thick, fill=white] (v7) at (2, -1) {$7^*$}; 
        \node[fill=white] (v8) at (3, 0) {8};
        \node[draw=purple, super thick, fill=white] (v8) at (3, 0) {8}; 
        \draw (v5) -- (v6) -- (v7) -- (v5);
        \draw (v6) -- (v8) -- (v7);
        \draw (v4) -- (v5);
        \draw[super thick, purple] (v5) -- (v7) -- (v8); 
        \node[label] at (0.5, -1.65) {Time step $t=0$: $C_x^0 = \{5, 7, 8\}$};

        \node[fill=white] (w1) at (4, 0) {1};
        \node[initial blue] (w2) at (5, 1) {2};
        \node[forced blue] (w3) at (5, -1) {3};
        \node[initial blue] (w4) at (6, 0) {4};
        \draw (w1) -- (w2) -- (w3) -- (w1);
        \draw (w2) -- (w4) -- (w3);

        \node[forced blue] (w5) at (7, 0) {5};
        \node[initial blue] (w6) at (8, 1) {6};
        \node[draw=purple, super thick, fill=white] (w7) at (8, -1) {$7^*$}; 
        \node[draw=purple, super thick, fill=white] (w8) at (9, 0) {8}; 
        \draw (w5) -- (w6) -- (w7) -- (w5);
        \draw (w6) -- (w8) -- (w7);
        \draw (w4) -- (w5);
        \draw[->, very thick, blue] (w4) -- (w5); 
        \draw[->, very thick, blue] (w4) -- (w3); 
        \draw[super thick, purple] (w7) -- (w8); 
        \node[label] at (6.5, -1.65) {Time step $t=1$: $C_x^1 = \{7, 8\}$};

    \begin{scope}[yshift=-0.42cm]
        \node[forced blue] (x1) at (-2, -3.5) {1};
        \node[initial blue] (x2) at (-1, -2.5) {2};
        \node[forced blue] (x3) at (-1, -4.5) {3};
        \node[initial blue] (x4) at (0, -3.5) {4};
        \draw (x1) -- (x2) -- (x3) -- (x1);
        \draw (x2) -- (x4) -- (x3);

        \node[forced blue] (x5) at (1, -3.5) {5};
        \node[initial blue] (x6) at (2, -2.5) {6};
        \node[forced blue] (x7) at (2, -4.5) {$7^*$}; 
        \node[fill=white] (x8) at (3, -3.5) {8}; 
        \draw (x5) -- (x6) -- (x7) -- (x5);
        \draw (x6) -- (x8) -- (x7);
        \draw (x4) -- (x5);
        \draw[->, very thick, blue] (x5) -- (x7); 
        \draw[->, very thick, blue] (x3) -- (x1); 
        \node[label] at (0.5, -5.15) {Time step $t=2$: $C_x^2 \text{ undefined}$};

        \node[fill=white] (z1) at (4, -3.5) {1};
        \node[initial blue] (z2) at (5, -2.5) {2};
        \node[fill=white] (z3) at (5, -4.5) {3};
        \node[initial blue] (z4) at (6, -3.5) {4};
        \draw (z1) -- (z2) -- (z3) -- (z1);
        \draw (z2) -- (z4) -- (z3);

        \node[forced blue] (z5) at (7, -3.5) {5};
        \node[initial blue] (z6) at (8, -2.5) {6};
        \node[forced blue] (z7) at (8, -4.5) {$7^*$}; 
        \node[fill=white] (z8) at (9, -3.5) {8}; 
        \draw (z5) -- (z6) -- (z7) -- (z5);
        \draw (z6) -- (z8) -- (z7);
        \draw (z4) -- (z5);
        \draw[->, very thick, blue] (z4) -- (z5);
        \draw[->, very thick, blue] (z5) -- (z7); 
        \draw(z3) -- (z1); 
        \node[label] at (6.5, -5.15) {Path bundle: $\mathcal{Q} = \{\{2\}, \{4, 5, 7\}, \{6\}\}$};
        \node[label] at (6.5, -5.5) {Terminus: $\Term(\F|_\mathcal{Q}) = \{2, 6, 7\}$};
    \end{scope}
    \end{tikzpicture}
    \caption{An example illustrating the construction of a path bundle induced by the vertex $x=7$. The purple vertices and edges highlight the component $C_x^t$ containing $x$ at each time step. }
    \label{fig:fixed-vertex-connected-components}
\end{figure}

Fix a vertex $x \in V(G)$, and suppose that $x$ becomes blue at time step $t_x$ of a relaxed chronology $\F$. For each $t \in \{0, \dots, t_x-1\}$, let $C_x^t$ denote the component of $G - E_\F^{[t]}$ containing $x$; see Figure~\ref{fig:fixed-vertex-connected-components} for an illustration. Using this sequence of components, we construct a collection of paths, denoted $\mathcal{Q}(\mathcal F; x)$, as follows. 
 Initially, let $\{v_i^0\}_{i=1}^{|B|}$ represent the vertices in the initial blue set $B$, and define $\mathcal{Q}^{[0]}(\mathcal F; x) = \{Q_i^{[0]}(\mathcal F; x)\}_{i=1}^{|B|}$, where each $Q_i^{[0]}(\mathcal F; x)$ is a single-vertex path consisting of $v_i^0$. At time step $t+1$, if a vertex $v_i^t$ forces a vertex $w \in C_x^t$ (i.e., $\psdforce{v_i^t}{w} \in F^{(t+1)}$), we define $v_i^{t+1} = w$ and extend the path $Q_i^{[t]}(\mathcal F; x)$ by appending $v_i^{t+1}$. For all other paths in $\mathcal{Q}^{[t+1]}(\mathcal F; x)$, $Q_i^{[t+1]}(\mathcal F; x)$ remains unchanged, with $v_i^{t+1} = v_i^t$. This process iterates until the final time step $t_x$, yielding the collection $\mathcal{Q}^{[t_x]}(\mathcal F; x)$.

The resulting collection of paths, denoted $\mathcal{Q}(\mathcal{F};x) = \mathcal{Q}^{[t_x]}(\mathcal F; x)$, is called the \emph{path bundle of $\F$ induced by $x$}. If no specific vertex $x$ is mentioned, we refer to such structures more generally as \emph{path bundles}. The \emph{terminus} of the restriction $\F|_{\mathcal{Q}}$, denoted $\Term(\F|_{\mathcal{Q}})$, is the set of vertices in $G$ that do not perform any forces in $\F|_{\mathcal{Q}(\mathcal{F};x)}$. These vertices represent the endpoints of propagation paths within $\mathcal{Q}$, providing insight into how positive semidefinite forcing processes terminate within specific induced structures of $G$; see Figure~\ref{fig:fixed-vertex-connected-components} for an illustration.


\section{\textbf{THE LEMMAS}}\label{sec:the-lemmas}
The following result connects path bundles and positive semidefinite forcing sets.
\begin{lemma}[\cite{PIPs}]\label{lem:reverse}
Let $G$ be a graph, and let $B$ be a positive semidefinite forcing set of $G$. If $\mathcal{F}$ is a relaxed chronology of forces for $B$ on $G$, and $\mathcal{Q}$ is the path bundle of $\mathcal{F}$ induced by some vertex $x \in V(G)$, then the set $\Term(\mathcal{F}|_{\mathcal{Q}(\mathcal F; x)})$ is a positive semidefinite forcing set of $G$.
\end{lemma}

The next lemma shows that any minimum positive semidefinite forcing set of a connected graph can be modified to maximize the size of a component in its complement. 

\begin{lemma}\label{lem:disconnected}
Let $G$ be a connected graph, and let $S \subseteq V(G)$ be a minimum positive semidefinite forcing set of $G$. If the induced subgraph $G - S$ is disconnected, then given a component $C$ of $G-S$, there exists another minimum positive semidefinite forcing set $S' \subseteq V(G)$ such that $G - S'$ has a component $C'$ containing $C$ as a proper subgraph.
\end{lemma}

\begin{proof}
Let $G$ be a connected graph, and let $S \subseteq V(G)$ be a minimum positive semidefinite forcing set of $G$. Suppose $G - S$ is disconnected, and let $C$ be a component of $G - S$.  The only neighbors of vertices in $V(C)$ outside of $V(C)$ must lie in $S$. Thus, we define $S_0 = N(V(C)) \cap S$ as the set of these neighbors.

Observe that $S_0$ forms a vertex cut in $G$ that separates $V(C)$ from $V(G) \setminus (S_0 \cup V(C))$. Specifically, any path in $G$ between a vertex in $V(C)$ and a vertex in $V(G) \setminus (S_0 \cup V(C))$ must pass through $S_0$. Furthermore, since $G$ is connected, there must exist a vertex $x \in S_0$ such that $x$ is adjacent to at least one vertex in $V(G) \setminus (S_0 \cup V(C))$. If no such vertex exists, the subgraph induced by $S_0 \cup V(C)$ would form a component of $G$, contradicting the assumption that $G$ is connected. 

Now, let $\mathcal{F}$ be a chronological list of forces for $S$ with expansion sequence $\{E_{\mathcal{F}}^{[i]}\}_{i=0}^\tau$.  Consider the progression of the set of vertices $E_{\mathcal{F}}^{[i]}$ that are blue at time step $i$.  At some time step $t$, every neighbor of $x$ not belonging to $C$ will have been colored blue; we choose $t$ to be the first such time step.  Since $x\in S$ is initially blue, for our chosen $t$ we have that
\[
N_G[x] \subseteq E_{\mathcal{F}}^{[t]} \cup V(C).
\]

We claim that $t > 0$. To show this, suppose instead that $t = 0$. If $t = 0$, then $E_{\mathcal{F}}^{[0]} = S$, so $N_G[x] \subseteq S \cup V(C)$. Since $x$ is adjacent to at least one vertex in $V(G) \setminus (S_0 \cup V(C))$, there must exist a vertex $y \in N(x)$ such that $y \in V(G) \setminus (S_0 \cup V(C))$. By definition, the vertex $y$ belongs to both $S \cup V(C)$, which includes all vertices that are initially blue or in $V(C)$, and $V(G) \setminus (S_0 \cup V(C))$, which excludes the vertices in $S_0$ and $V(C)$. Thus, $y \in S \setminus S_0$.

Define a new set $S' = S \setminus \{y\}$ by removing the vertex $y$ from $S$. Color the vertices in $S'$ blue and those in $V(G) \setminus S'$ white. Observe that the component of $G - S'$ containing $y$ intersects the neighborhood of $x$ only at $y$ itself. This implies that $\psdforce{x}{y}$ is a valid force at time $t = 0$, with $S'$ as the set of initially blue vertices. Hence, $x$ can force $y$ to become blue under the positive semidefinite forcing rule. After this color change, the resulting set of blue vertices is precisely the original set $S$. Since $S$ is a positive semidefinite forcing set, $S'$ must also be a positive semidefinite forcing set of $G$. However, this contradicts the assumption that $S$ is a minimum positive semidefinite forcing set, as $|S'| < |S|$. Thus, $t > 0$.

Therefore, we must have that $t > 0$. Let $w^*$ denote the neighbor of $x$ forced at time step $t$ in the chronological list of forces $\mathcal{F}$. Since $t > 0$, $w^*$ is not initially blue but becomes blue during $\mathcal F$. The component $C_{w^*}^{t-1}$ of $G - E_{\mathcal{F}}^{[t-1]}$ containing $w^*$ intersects $N(x)$ only at $w^*$, i.e., $C_{w^*}^{t-1} \cap N(x) = \{w^*\}$. Thus, the force $\psdforce{x}{w^*}$ is valid for $E_{\mathcal{F}}^{[t-1]}$. 

Define a new chronological list of forces $\mathcal{F}'$ by applying the same forces as in $\mathcal{F}$, except at time step $t$. At time step $t$, instead of applying the original force $\psdforce{v}{w^*}$, apply the force $\psdforce{x}{w^*}$. In this way, $\mathcal{F}'$ agrees with $\mathcal{F}$ at all time steps except possibly $t$, where the forcing vertex is required to be $x$.  Since $\psdforce{x}{w^*}$ is valid for $E_{\mathcal{F}}^{[t-1]}$, and since $S$ is a positive semidefinite forcing set, it follows that all of $V(G)$ will still become colored blue under the chronological list of forces $\mathcal{F}'$.

Next, consider the path bundle $\mathcal{Q}(\mathcal{F}'; w^*)$, capturing how $w^*$ becomes blue through the forcing process. Restrict $\mathcal{F}'$ to the forces occurring within this bundle, and let $S' = \Term(\mathcal{F}'|_{\mathcal{Q}(\mathcal{F}'; w^*)})$ denote the terminus of this restriction. By Lemma~\ref{lem:reverse}, $S'$ is a minimum positive semidefinite forcing set of $G$. Finally, observe that $C_{w^*}^0$, the component of $G - S$ containing $w^*$, is entirely contained within $V(G) \setminus (S \cup V(C))$. By definition, the sets $V(G) \setminus (S \cup V(C))$ and $V(C)$ are disjoint, so no vertex of $C_{w^*}^0$ can belong to $V(C)$. Thus, $S' \cap V(C) = \emptyset$, as $S'$ is the terminus of the restricted forces in $G[V(C_{w^*}^0) \cup S]$. Additionally, $x \notin S'$, as $x$ participates in forcing and thus does not belong to the terminus. Consequently, all vertices in $V(C)$ and $x$ belong to the same component of $G - S'$. Specifically, the component of $V(G) - S'$ containing $x$ strictly contains $C$, satisfying $V(C) \subsetneq V(C')$. This completes the proof.
\end{proof}

\begin{lemma}\label{lem:connected}
If $G$ is a connected graph, then there exists a minimum positive semidefinite forcing set $S \subseteq V(G)$ such that $G - S$ is connected.
\end{lemma}
\begin{proof}
Let $S \subseteq V(G)$ be an arbitrary minimum positive semidefinite forcing set of $G$. If $G - S$ is connected, then $S$ satisfies the conclusion of the lemma, and we are done. Suppose instead that $G - S$ is disconnected. By Lemma~\ref{lem:disconnected}, given a component $C$ of $G-S$, there exists another minimum positive semidefinite forcing set $S' \subseteq V(G)$ such that $G - S'$ has a strictly larger component than $C$. Specifically, there exists a component $C'$ of $G - S'$ that contains $C$ with $V(C) \subsetneq V(C')$. 

Replace $S$ with $S'$ and $C$ with $C'$, and repeat the process. At each step, the size of the component $C$ in $G - S$ strictly increases. Since $G$ is a finite graph, this process must terminate after finitely many steps. When the process terminates, we obtain a minimum positive semidefinite forcing set $S$ such that $G - S$ is connected,  completing the proof.
\end{proof}


\section{\textbf{THE MAIN THEOREM AND COROLLARY}}\label{sec:the-theorem}
This section presents our main result, Theorem~\ref{thm:main}, and its corollary, Corollary~\ref{cor:main}. 
\begin{theorem}\label{thm:main}
If $G$ is a connected and claw-free graph, then $Z_+(G) = Z(G)$.
\end{theorem}
\begin{proof}
Let $G$ be a connected and claw-free graph. Since every standard forcing set is also a positive semidefinite forcing set, it follows that $Z_+(G) \leq Z(G)$. To prove the theorem, it suffices to show that $Z_+(G) \geq Z(G)$. By Lemma~\ref{lem:connected}, there exists a minimum positive semidefinite forcing set $S \subseteq V(G)$ such that the induced subgraph $G[W]$ is connected, where $W = V(G) \setminus S$. Fix such a set $S$, and color the vertices in $S$ blue and the vertices in $W$ white.

We aim to show that every valid positive semidefinite force during a chronological list of forces beginning at $S$ is also a valid standard force, which will, in turn, imply that $S$ is a standard forcing set of $G$. Along the way we will need to track that the white colored vertices always induce a connected subgraph; that is, the set of white colored vertices at time $t$ induces a connected subgraph of $G$, for each time step $0\leq t<\tau$.

\medskip
\noindent\textbf{Base Case ($t = 0$):} At time $t = 0$, the set of white vertices is $W^{[0]} = W$, and by the choice of $S$, the induced subgraph $G[W^{[0]}]$ is connected.  Also, since no forces occur prior to $t=1$, it is vacuously true that each force prior to that point is a valid standard force. Next consider the first positive semidefinite force $\psdforce{u}{v}$ that occurs at time $t = 1$. By the positive semidefinite forcing rule, $u$ is a blue vertex with $v$ as its unique white neighbor in the component of $G[W^{[0]}]$ containing $v$. Since $G[W^{[0]}]$ is connected, this implies that $v$ is the unique white neighbor of $u$ in $G$. Thus, $G[W^{[0]}]$ is connected and any positive semidefinite force at time $t = 1$ is also a valid standard force at time $t = 1$. 

\medskip
\noindent\textbf{Inductive Step:} For our inductive step, we assume that for some $t> 0$, we have that
\begin{enumerate}
    \item[(a)] The induced subgraph $G[W^{[t-1]}]$ is connected, where $W^{[t-1]}$ is the set of white vertices at time $t-1$.
    \item[(b)] Every positive semidefinite force up to and including time $t$ corresponds to a valid standard force in $G$.
\end{enumerate}

Let $v$ be the vertex that was forced to become blue at time $t$ by the vertex $u$. We first verify that the induced subgraph of the remaining white vertices, $G[W^{[t]}]$, remains connected after the force at time $t$ has been applied, where
\[
W^{[t]} = W^{[t-1]} \setminus \{v\}.
\]
Suppose, for contradiction, that $G[W^{[t]}]$ is disconnected. This implies that the removal of $v$ from $G[W^{[t-1]}]$ disconnects the white colored subgraph. Thus, $v$ is a cut-vertex of $G[W^{[t-1]}]$. Consequently, $v$ has at least two white neighbors in $G[W^{[t-1]}]$ that belong to different components of $G[W^{[t]}]$, say $w$ and $z$.  Since $\zforce{u}{v}$ is a valid standard force, $v$ is the unique white neighbor of $u$ in $G$ at time $t-1$.  Therefore, $u$ is not adjacent to any of the white neighbors of $v$, and in particular, the nonadjacent vertices $w$ and $z$. Thus, the set $\{u, v, w, z\}$ induces a claw centered at $v$, contradicting the assumption that $G$ is claw-free. Therefore, $G[W^{[t]}]$ must be connected. 

Next suppose that $\psdforce{u'}{v'}$ is a valid positive semidefinite force at time $t+1$. By the positive semidefinite forcing rule, $u'$ is a blue vertex at time $t$, and $v'$ is the unique white neighbor of $u'$ in the component of $G[W^{[t]}]$ containing $v'$. Since $G[W^{[t]}]$ is connected, it follows that $v'$ is the unique white neighbor of $u'$ in $G$. Therefore, $\zforce{u'}{v'}$ is a valid standard force in $G$ at time $t+1$.

\medskip
By induction, we have established that at each time step $t \geq 1$, the positive semidefinite force $\psdforce{u}{v}$ occurring during the chronological list of forces starting from $S$ is also a valid standard force. Therefore, the chronological list of positive semidefinite forces starting from $S$ effectively mirrors a standard forcing process. Since $S$ is a positive semidefinite forcing set, this process eventually colors all vertices of $G$ blue, which implies that $S$ is a standard forcing set of $G$. Thus, we have 
\[
Z(G) \leq |S| = Z_+(G),
\]
which when combined with the earlier inequality $Z_+(G) \leq Z(G)$, implies $Z_+(G) = Z(G)$, our desired result. 
\end{proof}

Since the standard and positive semidefinite forcing numbers are additive with respect to components, it is clear that any disconnected and claw-free graph $G$ satisfies $Z_+(G) = Z(G)$. Using this fact, together with induced subgraphs of claw-free graphs being claw-free, we now give a characterization of claw-free graphs in terms of standard and positive semidefinite forcing. 

\begin{corollary}\label{cor:main}
A graph is $(Z_+, Z)$-perfect if and only if it is claw-free. 
\end{corollary}
\begin{proof}
Let $G$ be a graph. If $G$ is claw-free, then any induced subgraph $H$ of $G$ is also claw-free. Thus, by Theorem~\ref{thm:main}, $Z_+(H) = Z(H)$ for every induced subgraph $H$ of $G$, and $G$ is $(Z_+, Z)$-perfect. 

Conversely, and proceeding by way of contraposition, suppose $G$ is a graph containing at least one induced subgraph $H \subseteq G$ with $H \cong K_{1, 3}$. Then, 
\[
1 = Z_+(H) < Z(H) = 2,
\]
implying that $G$ is not $(Z_+, Z)$-perfect. Thus, 
$G$ is $(Z_+, Z)$-perfect if and only if $G$ is claw-free.
\end{proof}


\section{\textbf{CONCLUSION}}\label{sec:conclusion}
In this paper, we proved that $Z_+(G) = Z(G)$ for all claw-free graphs (Theorem~\ref{thm:main}), thereby confirming Conjecture~\ref{conj:main} proposed by \emph{TxGraffiti}. As a consequence, we established that a graph $G$ is $(Z_+, Z)$--perfect if and only if $G$ is claw-free (Corollary~\ref{cor:main}). This result not only characterizes $(Z_+, Z)$--perfect graphs but also unifies and extends prior work on the relationship between $Z(G)$ and $Z_+(G)$.

Our characterization naturally leads to several questions for future research. Can similar equivalences be identified for other graph classes or for other variants of zero forcing? What are the algorithmic consequences of $(Z_+, Z)$--perfectness, particularly for claw-free graphs? Moreover, how might these results generalize to directed graphs or other settings? Investigating these questions may provide further insights into the connections between structural graph theory and zero forcing concepts.

\bibliographystyle{amsplain}
\bibliography{references}

\end{document}